\documentclass[12pt,sort&compress]{elsarticle}

\journal{Journal of Mathematical Analysis and its Applications}

\usepackage{amsmath,amsthm,amsfonts}
\usepackage{dsfont}

\newcommand{\E}{\mathbf E}
\newcommand{\R}{\mathbb R}

\newcommand{\eps}{\varepsilon}
\newcommand{\Sphere}{{\mathbb S}^{d-1}}
\newcommand{\symdif}{\bigtriangleup}
\newcommand{\thf}{\frac{1}{2}}

\newcommand{\ind}{\mathds{1}}

\DeclareMathOperator{\Vol}{V}
\DeclareMathOperator{\pr}{pr}

\theoremstyle{plain} \newtheorem{theorem}{Theorem}[section]
\theoremstyle{plain} \newtheorem{proposition}[theorem]{Proposition}
\theoremstyle{plain} \newtheorem{lemma}[theorem]{Lemma}
\theoremstyle{plain} \newtheorem{corollary}[theorem]{Corollary}
\theoremstyle{definition} \newtheorem{definition}[theorem]{Definition}
\theoremstyle{definition} 
\theoremstyle{remark} 
\theoremstyle{remark} \newtheorem{example}[theorem]{Example}

\usepackage{fancybox}
\newlength{\querylen}
\begin{document}

\begin{frontmatter}

\title{A generalisation of the fractional Brownian field based on non-Euclidean norms}

\author[mainaddress]{Ilya Molchanov}
\ead{ilya.molchanov@stat.unibe.ch}

\author[secondaryaddress]{Kostiantyn Ralchenko}
\ead{k.ralchenko@gmail.com}

\address[mainaddress]{University of Bern,
 Institute of Mathematical Statistics and Actuarial Science,
 Sidlerstrasse 5,
 CH-3012 Bern,
 Switzerland}

\address[secondaryaddress]
{Taras Shevchenko National University of Kyiv,
 Department of Probability Theory, Statistics and Actuarial Mathematics,
 Volodymyrska 64/13,
 01601 Kyiv,
 Ukraine}

\begin{abstract}
  We explore a generalisation of the L\'evy fractional Brownian field
  on the Euclidean space based on replacing the Euclidean norm with
  another norm. A characterisation result for admissible norms yields
  a complete description of all self-similar Gaussian random fields
  with stationary increments. Several integral representations of the
  introduced random fields are derived.

  In a similar vein, several non-Euclidean variants of the fractional
  Poisson field are introduced and it is shown that they share the
  covariance structure with the fractional Brownian field and converge
  to it.  The shape parameters of the Poisson and Brownian variants
  are related by convex geometry transforms, namely the radial $p$th
  mean body and the polar projection transforms.
\end{abstract}


\end{frontmatter}

\section{Introduction}
\label{sec:introduction}

The \emph{multiparameter fractional Brownian motion} or the
\emph{L\'evy fractional Brownian field} (fBf) with Hurst index
$H\in(0,1)$ is a centred Gaussian random field $X(z)$, $z\in\R^d$,
with the covariance function
\begin{equation}
  \label{eq:1}
  \E[X(z_1)X(z_2)]=\thf\left[\|z_1\|^{2H}+\|z_2\|^{2H}-\|z_1-z_2\|^{2H}\right]\,,
\end{equation}
where $\|z\|$ is the Euclidean norm of $z$.  If $h=\thf$, this yields
the L\'evy Brownian motion on $\R^d$.  If $d=1$, one recovers the
classical univariate fractional Brownian motion (fBm), see
\cite{MandelbrotVanNess68}.  This random field was introduced by
A.M.~Yaglom \cite{yag57} as a model of turbulence in fluid
mechanics. Various proofs showing that \eqref{eq:1} defines a valid
covariance function are given in
\cite{Gangolli67,MandelbrotVanNess68,oss:way89,sam:taq94}.  Further
results including series expansions and a general functional limit
theorem can be found in \cite{Malyarenko13}.  The two most important
integral representations of the L\'evy fBf are the moving average
representation using the integral with respect to the white noise and
the harmonisable representation as an integral with respect to the
Fourier transform of the white noise, see
\cite{coh:ist13,herb06,lind93,sam:taq94}.

Istas \cite{ist05} defined the fractional Brownian motion $B$ on a
metric space by assuming that the square of its increment $B(x)-B(y)$
is normally distributed with the variance given by the $2H$-power of
the metric distance between $x$ and $y$. The existence of fractional
Brownian motions on the Euclidean sphere and on the hyperbolic space
is verified for $H\in(0,\thf]$. These constructions have been extended
to stable random fields in \cite{ist06}. See also the recent monograph
\cite{coh:ist13} for a number of results on general self-similar
random fields.

Bierm\'e et al. \cite{bier:dem:est13} considered a random field
generated by a Poisson random measure on $\R^d\times[0,\infty)$ and
proved that it shares the same covariance function \eqref{eq:1} with
the L\'evy fBf. Such a field may be called a
fractional Poisson field, noticing that other definitions of
fractional Poisson fields are available in the literature, see
e.\,g.\ \cite{meer:nan:vel11} for the univariate case and
\cite{leon:mer13} for a multivariate generalisation.

In this paper we introduce a generalisation of the fBf based on
replacing the Euclidean norm in \eqref{eq:1} with a non-Euclidean
one. The space $\R^d$ with such norm is called the Minkowski space
\cite{thom96}, so that we term our generalisation the Minkowski
fractional Brownian field (MfBf). Section~\ref{sec:norms-star-bodies}
introduces necessary concepts from convex geometry. In
Section~\ref{sec:gener-fract-brown} we establish that the norms giving
rise to valid covariance functions are generated by $L_p$-balls
related to the isometric embeddability of the Minkowski space into
$L_p([0,1])$ for $p=2H$. Furthermore, we derive several integral
representations of the introduced random field. In addition to
conventional integral representations based on integrating the white
noise or its Fourier transform, we derive novel representations based
on sums of series of L\'evy fBf's and integrals of univariate
fractional Brownian motions. Furthermore, we relate the ordering of
expected supremum of MfBf with the Banach--Mazur distance between
normed spaces. The key idea is the equivalence relation on the family
of MfBf up to non-degenerate linear transformations of their
arguments.

Section~\ref{sec:fract-poiss-fields} introduces random fields based on
Poisson point processes that share the covariance function with the
MfBf for $H\in(0,\thf)$. The construction follows the ideas from
\cite{bier:dem:est13} and \cite{wan:wen03}, and is also related to
random balls models \cite{bret:dom09} and the studies of micropulses
\cite{mar11}. In difference to the previous works, we emphasise the
role of the shape parameters of the corresponding fields. The main
result provides a relationship between the shape parameters of the
Poisson random field and its Gaussian counterpart. This relationship
is given by the radial $p$th mean body transformation introduced in
\cite{gar:zhan98}. The convergence to the Brownian field with $H=\thf$
using different normalisations of the Poisson model is considered in
Section~\ref{sec:conv-mfbf-with}. These results are new even in the
case of L\'evy fBf. The shape parameters are related by the polar
projection body transform known from the convex geometry
\cite{schn2}. Finally Section~\ref{sec:other-constr-fract} presents
other constructions of the fractional Poisson fields that share the
covariance structure with the MfBf.

\section{Norms and star bodies}
\label{sec:norms-star-bodies}

A closed bounded set $F$ in $\R^d$ is called a \emph{star body} if for
every $u\in F$ the interval $\{tu:\; 0\leq t<1\}$ is contained in the
interior of $F$ and the \emph{Minkowski functional} (or the gauge
function) of $F$ defined by
\begin{displaymath}
  \|u\|_F=\inf\{s\geq0:\; u\in sF\}
\end{displaymath}
is a continuous function of $u\in\R^d$.   The set $F$ can be
recovered from its Minkowski functional by
\begin{displaymath}
  F=\{u:\; \|u\|_F\leq 1\}\,,
\end{displaymath}
while the \emph{radial function}
\begin{math}
  \rho_F(u)=\|u\|_F^{-1}
\end{math}
provides the polar coordinate representation of the boundary of $F$
for $u$ from the unit \emph{Euclidean sphere} $\Sphere$.  In the
following we mainly consider origin-symmetric star-shaped sets and call them
\emph{centred} in this case. If the star body $F$ is centred and
\emph{convex}, then $\|u\|_F$ becomes a convex norm on $\R^d$ and
$(\R^d,\|\cdot\|_F)$ is called a \emph{Minkowski space}, see
\cite{thom96}. We also keep the same notation $\|u\|_F$ if the norm is
not convex.  Further $\|x\|$ (without subscript) denotes the
\emph{Euclidean norm} of $x\in\R^d$. By $\Vol_d(K)$ we denote the
$d$-dimensional \emph{Lebesgue measure} of a measurable set $K$.

The \emph{$p$-sum} of two star bodies $F_1$ and $F_2$ is the star body
$F$ such that
\begin{displaymath}
  \|u\|_F^p=\|u\|_{F_1}^p+\|u\|_{F_2}^p\,,
\end{displaymath}
see \cite{kold05}. 

The \emph{support function} of a convex set $K$ in $\R^d$ is defined
by
\begin{displaymath}
  h(K,u)=\sup\{\langle x,u\rangle:\; x\in K\}\,,\quad u\in\R^d\,.
\end{displaymath}
Note that the support function may take infinite values if $K$ is not
bounded.  The \emph{polar set} to a convex set $K$ containing the
origin is defined by
\begin{displaymath}
  K^*=\{u:\; h(K,u)\leq1\}\,.
\end{displaymath}

An \emph{$L_p$-ball} for $p>0$ is a star body $F$ such that
$(\R^d,\|\cdot\|_F)$ is isometrically embeddable into $L_p([0,1])$,
see \cite[Lemma~6.4]{kold05}. The star body $F$ is an $L_p$-ball if
and only if
\begin{equation}
  \label{eq:spectral}
  \|z\|_F^p=\int_{\Sphere} |\langle z,u\rangle|^p \sigma(du)
\end{equation}
for a finite even measure $\sigma$ on the unit Euclidean sphere
$\Sphere$, see \cite[Lemma~4.8]{grin:zhan99}. 
The $L_p$-balls are necessarily convex for $p\in[1,2]$.

\begin{example}
  \label{ex:balls}
  A set $F$ is an $L_2$-ball if and only if
  \begin{displaymath}
    \|z\|_F^2=\int_{\Sphere} |\langle z,u\rangle|^2 \sigma(du)
    =\langle Az,z\rangle\,,
  \end{displaymath}
  where the matrix $A$ is symmetric non-negative definite with entries
  given by
  \begin{displaymath}
    a_{ij}=\int_{\Sphere} u_iu_j \sigma(du)\,, \quad i,j=1,\dots,d\,.
  \end{displaymath}
  Thus, the family of $L_2$-balls coincides with the family of
  ellipsoids.
\end{example}

\section{Minkowski fractional Brownian field}
\label{sec:gener-fract-brown}

\subsection{Definition and existence}
\label{sec:definition}

\begin{definition}
  \label{def:f-fbm}
  A centred Gaussian random field $X_F(z)$, $z\in\R^d$, with the
  covariance function $C_F(z_1,z_2)=\E[X_F(z_1)X_F(z_2)]$ given by
  \begin{equation}
    \label{eq:cov-F}
    C_F(z_1,z_2)=\thf
    \left[\|z_1\|_F^{2H}+\|z_2\|_F^{2H}-\|z_1-z_2\|_F^{2H}\right],
    \qquad z_1,z_2\in\R^d\,,
  \end{equation}
  is called the \emph{Minkowski fractional Brownian field} (MfBf) with
  the Hurst parameter $H\in(0,1]$ and the \emph{associated star body}
  $F$.
\end{definition}

Since in dimension $d=1$ all Minkowski functionals are identical up to
a constant, a non-trivial generalisation is only possible if the
dimension $d$ is at least $2$.

The increment of the MfBf is centred Gaussian with the variance
\begin{displaymath}
  \E(X_F(z_1)-X_F(z_2))^2=\|z_1-z_2\|_F^{2H}\,.
\end{displaymath}
Thus, the MfBf is a Gaussian field with stationary increments.
It follows from \cite{dob79} that the family of MfBf coincides with
the family of all Gaussian random fields with stationary increments.

The following well-known result is useful to establish the positive
definiteness of the covariance \eqref{eq:cov-F}. Note that the
positive definiteness is understood in the non-strict sense.

\begin{lemma}
  \label{lemma:lem1}
  Let $f:\R^d\to\R_+$ be an even function. The function $e^{-cf(z)}$,
  $z\in\R^d$, is positive definite for all $c>0$ if and only if the
  function
  \[
  A_f(z_1,z_2)=f(z_1)+f(z_2)-f(z_1-z_2)
  \]
  is positive definite.
\end{lemma}
\begin{proof}
  If $t_1,\dots,t_k\in\R$ and $z_1,\dots,z_k\in\R^d$, then
  \begin{displaymath}
    \sum_{i,j=1}^k A_f(z_i,z_j)t_it_j
    =-\sum_{i,j=0}^k f(z_i-z_j)t_it_j\,,
  \end{displaymath}
  where $t_0=-\sum_{i=1}^k t_i$ and $z_0=0$ is the origin. Thus, $A_f$
  is positive definite if and only if $f$ is negative definite. The
  latter is equivalent to the positive definiteness of $e^{-cf}$ for
  all $c>0$, see \cite[Th.~2.2]{ber:c:r}.
\end{proof}

\begin{proposition}
  \label{prop:fbm-exist}
  The MfBf exists if and only if $H\in(0,1]$ and $F$ is an $L_p$-ball
  with $p=2H$.
\end{proposition}
\begin{proof}
  It is known \cite[Th.~6.6]{kold05} that $e^{-c\|z\|_F^p}$ is
  positive definite if and only if $(\R^d,\|\cdot\|_F)$ isometrically
  embeds in $L_p([0,1])$, meaning that $F$ is an $L_p$-ball. By
  considering $e^{-c|t|^p\|z\|_F}$ as a function of $t\in\R$ with a
  fixed $z$, it is easily seen that $p\in(0,2]$. The rest follows from
  Lemma~\ref{lemma:lem1}.
\end{proof}

Since the associated star body $F$ is an $L_p$-ball, the norm
$\|u\|_F$ admits representation \eqref{eq:spectral}, and the measure
$\sigma$ from \eqref{eq:spectral} is called the \emph{spectral
  measure} of the corresponding MfBf $X_F(\cdot)$.

\begin{example}[L\'evy fBf]
  \label{ex:levy-fbf}
  If $\sigma$ is the rotation invariant measure on the unit sphere
  with the total mass
  \begin{displaymath}
    \sigma(\Sphere)=\frac{\Gamma(H+\frac{d}{2})}{2\pi^{(d-1)/2}\Gamma(H+\thf)}\,,
  \end{displaymath}
  then $F$ is the unit Euclidean ball, and we recover the L\'evy fBf.
\end{example}

If $F$ is an $L_p$-ball with $p\in(0,2]$, then $F$ is also an
$L_r$-ball for all $r\in(0,p]$, see \cite[Cor.~6.7]{kold05}. Thus, if
the MfBf with the associated star body $F$ exists for some Hurst
index $H$, then it exists for all Hurst indices $H'\in(0,H]$.

\begin{example}
  Let 
  \begin{displaymath}
    F=\{x=(x_1,\dots,x_d)\in\R^d:\; |x_1|^p+\cdots+|x_d|^p\leq1\},
  \end{displaymath}
  be the centred $\ell_p$-ball in $\R^d$ with $d\geq2$ for
  $p\in(0,2]$. The MfBf with the associated star body $F$ exists if
  and only if $H\in(0,\thf p]$. For $d=2$, the MfBf exists for any
  convex centred star body $F$ and $H\in(0,\thf]$, see also
  \cite{kol91s}.  Indeed, it is well known \cite[Cor.~3.5.7]{schn2}
  that all centred convex bodies in the plane are $L_1$-balls and so
  $L_p$-balls for $p\in(0,1]$.
\end{example}

\begin{example}
  If $H=1$, then $F$ is an $L_2$-ball, so $F$ is necessarily an
  ellipsoid that corresponds to the quadratic form determined by
  matrix $A$, see Example~\ref{ex:balls}. In this case $X_F(z)=\langle
  z,\xi\rangle$ for centred normally distributed random vector $\xi$
  with the covariance matrix $A$. If $H\in(0,1)$ and $F=\{z:\; \langle
  Az,z\rangle\leq 1\}$ is an ellipsoid with a strictly positive
  definite matrix $A$, then $X_F(A^{-1/2}z)$ is the L\'evy fBf with the
  covariance given by \eqref{eq:1}.
\end{example}

\begin{example}
  The family of $L_1$-balls is the family of polar bodies to zonoids,
  well-known from convex geometry \cite[Sec.~3.5]{schn2}. Thus, the
  family of all Minkowski Brownian fields (that appear for $H=\thf$)
  corresponds to the family of zonoids.
\end{example}

Since $F$ is an $L_p$-ball with $p=2H$, \eqref{eq:spectral} implies
that $\|z\|_F^{2H}\leq c\|z\|^{2H}$ for a constant $c$, i.e. the
variance of the increments of the MfBf is bounded (up to a constant)
by that of the L\'evy fBf. Therefore, the MfBf inherits the local
properties from the L\'evy fBf with the same Hurst parameter, in
particular it is a.s. continuous.

\subsection{Integral representations}
\label{sec:integr-repr}

Unless $F$ is the Euclidean ball, it is not possible to use the
arguments based on the rotational invariance (like in
\cite{herb06,lind93}) to derive integral representations of the MfBf.

Let $m$ be the measure on $\R^d$ whose polar representation has the
directional component $\sigma(du)$ (being the spectral measure from
\eqref{eq:spectral}) and the radial component
$(2\pi)^{-d/2}r^{d-1}dr$. Consider a Gaussian measure $W_\sigma$ on
$\R^d$ with the control measure $m$, so that, for each square
integrable function $f$,
\begin{align*}
  \E\left(\int_{\R^d} f(x) W_\sigma(dx)\right)^2
  &=\int_{\R^d} f(x)m(dx)\\
  &=\frac{1}{(2\pi)^{d/2}}\int_{\Sphere}
  \int_0^\infty f(ru) r^{d-1}dr\sigma(du)\,.
\end{align*}
If $\sigma$ is the surface area measure on the unit sphere (the
$(d-1)$-dimensional Hausdorff measure), then $W_\sigma$ is the
conventional Brownian random measure as considered in
\cite[Sec.~2.1.6.1]{coh:ist13} (also called the white noise) up to a
multiplicative constant. The Fourier transform $\hat{W}_\sigma$ of
$W_\sigma$ is defined in the sense of generalised functions as in
\cite[Def.~2.1.16]{coh:ist13}.

\begin{theorem}
  \label{thr:harm}
  The MfBf with the spectral measure $\sigma$ is given by
  \begin{equation}
    \label{eq:hr}
    X_F(z)=a_{H,d}
    \int_{\R^d} \frac{e^{\imath \langle z,y\rangle}-1}{\|y\|^{H+d/2}}
    \hat{W}_\sigma(dy)\,,
  \end{equation}
  where
  \begin{displaymath}
    a_{H,d}=2(2\pi)^{\frac{d-1}{4}}(H\Gamma(2H)\sin (H\pi))^{\thf}\,.
  \end{displaymath}
\end{theorem}
\begin{proof}
  By passing to the polar coordinates and noticing that the measure
  $\sigma$ is even,
  \begin{align*}
    \int_{\R^d} \frac{1-e^{\imath \langle z,y\rangle}}{\|y\|^{2H+d}} m(dy)
    &=\frac{1}{(2\pi)^{d/2}}\int_{\Sphere}\int_0^\infty \left(1-e^{\imath t\langle z,u\rangle}\right)
    t^{-2H-1}dt \sigma(du)\\
    &=\frac{1}{(2\pi)^{d/2}}\int_{\Sphere}\int_0^\infty (1-\cos(t\langle z,u\rangle))
    t^{-2H-1}dt \sigma(du)\\
    &=\frac{c_H}{(2\pi)^{d/2}}\int_{\Sphere} |\langle z,u\rangle|^{2H} \sigma(du)\,,
  \end{align*}
  where
  \begin{displaymath}
    c_H=\int_0^\infty (1-\cos t)t^{-2H-1}dt
    =\frac{\pi}{4H\Gamma(2H)\sin (H\pi)} \,,
  \end{displaymath}
  see \cite[p.~329]{sam:taq94}.  The rest of the proof is carried over
  similarly to \cite[Prop.~2]{herb06}. The normalising constant
  $a_{H,d}$ is derived from the condition that
  $a_{H,d}^2c_H(2\pi)^{-d/2}=\thf$.
\end{proof}

\begin{corollary}
  The MfBf can be represented as
  \begin{equation}
    \label{eq:ir}
    X_F(z)=a_{H,d}b_{H,d}\int_{\R^d}[\|z-u\|^{H-d/2}-\|u\|^{H-d/2}]W_\sigma(du)
  \end{equation}
  for
  \begin{displaymath}
    b_{H,d}=2^{-H}
    \frac{\Gamma\left(\thf(-H+d/2)\right)}
    {\Gamma\left(\thf(H+d/2)\right)}\,.
  \end{displaymath}
\end{corollary}
\begin{proof}
  It suffices to note that the integrand from \eqref{eq:ir} is the
  Fourier transform of the integrand from \eqref{eq:hr} up to the
  constant $b_{H,d}$, see \cite{gel:shil64} noticing that the Fourier
  transform is defined with the factor $(2\pi)^{-d/2}$.
\end{proof}

\begin{example}
  Let the spectral measure $\sigma$ attach the masses $\thf$ to the
  points $\{\pm v\}$ for a fixed $v\in\Sphere$. Then $\|z\|_F=|\langle
  z,v\rangle|$. The corresponding covariance function
  \begin{displaymath}
    C_v(z_1,z_2)=\thf\left[|\langle z_1,v\rangle|^{2H}
    +|\langle z_2,v\rangle|^{2H}
    -|\langle z_1-z_2,v\rangle|^{2H}\right]
  \end{displaymath}
  is positive definite for all $H\in(0,1]$ and defines the MfBf
  $Y_v$. It is easy to see that $Y_v(z)=B_H(\langle z,v\rangle)$ for
  the univariate fBm $B_H$.
\end{example}

\begin{proposition}
  Let $B_H(t)$ be the fBm on the real line and let $M$ be an
  independent of $B_H$ Gaussian white noise on $\Sphere$ with the
  control measure $\sigma$. Then
  \begin{equation}
    \label{eq:plain-wave}
    X_F(z)=\int_{\Sphere} B_H(\langle z,v\rangle)M(dv)
  \end{equation}
  is the MfBf with the spectral measure $\sigma$.
\end{proposition}
\begin{proof}
  The covariance of $X_F(z)$ is
  \begin{align*}
    \E[X_F(z_1)&X_F(z_2)]
    =\E[\E[X_F(z_1)X_F(z_2)|B_H]]\\
    &=\E\int_{\Sphere} B_H(\langle z_1,v\rangle)
    B_H(\langle z_2,v\rangle)\sigma(dv)\\
    &=\thf \int_{\Sphere}\left[|\langle z_1,v\rangle|^{2H}
    +|\langle z_2,v\rangle|^{2H}
    -|\langle z_1-z_2,v\rangle|^{2H}\right]\sigma(dv)
  \end{align*}
  and so coincides with \eqref{eq:cov-F}.
\end{proof}

In particular, if $\sigma$ is the rotation invariant measure on
$\Sphere$ from Example~\ref{ex:levy-fbf}, then \eqref{eq:plain-wave} can be
viewed as the analogue of the plain-wave expansion of the norm
\cite[Sec.~I.3.10]{gel:shil64}. Using \eqref{eq:plain-wave}, the results
for the univariate fractional Brownian motion can be extended to the
multivariate setting. 

\subsection{Associated star bodies}
\label{sec:assoc-star-bodi}

Consider here the properties of the MfBf in relation to their
associated star bodies.

\begin{proposition}
  \label{prop:weak-conv}
  If $\{F_n, n\geq1\}$ is a sequence of $L_p$-balls with $p=2H$, such
  that $\|u\|_{F_n}\to \|u\|_F$ for a star body $F$, then $F$ is an
  $L_p$-ball and the finite-dimensional distributions of $X_{F_n}$
  converge to those of $X_F$. If, additionally, there exists $\eps>0$
  such that $F_n\supset \eps B_2^d$ for all sufficiently large $n$,
  then $X_{F_n}$ weakly converges to $X_F$ in the space of continuous
  functions on any compact subset of $\R^d$.
\end{proposition}
\begin{proof}
  The convergence of finite dimensional distributions follows
  from the convergence of covariance functions, and the
  positive definiteness of the limiting covariance implies that $F$ is
  an $L_p$-ball. Notice that 
  \begin{displaymath}
    \E (X_{F_n}(z_1)-X_{F_n}(z_2))^{2k}
    =c_1\|z_1-z_2\|_{F_n}^{2Hk}\leq c_2\|z_1-z_2\|^{2Hk}.
  \end{displaymath}
  for all $z_1,z_2$ from a compact subset of $\R^d$, constants
  $c_1,c_2$, and some $k\geq1$ such that $2kH>d$.  The weak
  convergence follows from the tightness condition from
  \cite{ibr:has81}, see also \cite[Th.~2]{dav:zit08}.
\end{proof}

\begin{proposition}
  \label{prop:slepian}
  For two MfBf's $X_{F_1}$ and $X_{F_2}$, we have
  \begin{equation}
    \label{eq:slepian}
    \E \sup_{z\in D}|X_{F_1}(z)|\geq \E \sup_{z\in D}|X_{F_2}(z)|
  \end{equation}
  for each compact set $D\subset\R^d$ if and only if $F_1\subset F_2$.
\end{proposition}
\begin{proof}
  The sufficiency follows from the Sudakov--Fernique inequality
  \cite[Th.~2.2.3]{adl:tay07} noticing that
  \begin{displaymath}
    \E(X_{F_1}(z_1)-X_{F_1}(z_2))^2
    =\|z_1-z_2\|_{F_1}^{2H} \geq \|z_1-z_2\|_{F_2}^{2H}
    =\E(X_{F_2}(z_1)-X_{F_2}(z_2))^2\,.
  \end{displaymath}
  For the reverse implication, consider $D=\{z\}$. Then
  $\E|X_{F_1}(z)|\geq \E|X_{F_2}(z)|$ implies $\E X_{F_1}(z)^2\geq \E
  X_{F_2}(z)^2$ and it remains to notice that the increments are
  stationary.
\end{proof}

Let $A\in\mathrm{GL}(d)$ be an invertible matrix. Then $X_{AF}(z)$ is
a version of $X_F(A^{-1}z)$, $z\in\R^d$. The MfBf's obtained by such
transformations may be regarded as equivalent in a certain sense. In
particular, all ellipsoids $F$ can be transformed to the unit ball in
this way, so that all MfBf's with elliptical associated star bodies
can be considered equivalent to the L\'evy fBf. 

In view of Proposition~\ref{prop:slepian}, the infimum of $t\geq1$
such that
\begin{displaymath}
  \E \sup_{z\in D}|X_{F_2}(z)|\geq
  \E \sup_{z\in D}|X_{F_1}(Az)|\geq \frac{1}{t}
  \E \sup_{z\in D}|X_{F_2}(z)|
\end{displaymath}
for some $A\in\mathrm{GL}(d)$ and for all compact sets $D$ in $\R^d$
equals the infimum of $t>0$ such that $F_2\subset AF_1\subset tF_2$
for some $A\in\mathrm{GL}(d)$, which is the \emph{Banach--Mazur
  distance} between the normed spaces $(\R^d,\|\cdot\|_{F_1})$ and
$(\R^d,\|\cdot\|_{F_2})$, see \cite[Sec.~2.1]{lin:mil93}. Since the
Banach--Mazur distance between $(\R^d,\|\cdot\|_F)$ and the Euclidean
space is at most $\sqrt{d}$ (see \cite{lin:mil93}), we deduce that for
each MfBf $X_F$ there is the L\'evy fBf $X$ such that
\begin{displaymath}
  \E \sup_{z\in D}|X(z)|\geq
  \E \sup_{z\in D}|X_F(Az)|\geq \frac{1}{\sqrt{d}}
  \E \sup_{z\in D}|X(z)|
\end{displaymath}
for some $A\in\mathrm{GL}(d)$ and all compact sets $D\subset\R^d$.

Equation~\eqref{eq:slepian} provides a possible ordering of Gaussian
processes. Another ordering (which is stronger than \eqref{eq:slepian}
for fields that vanish at the origin) is the \emph{convex ordering} of
all finite-dimensional distributions meaning that
\begin{displaymath}
  \E g(X_{F_1}(z_1),\dots,X_{F_1}(z_n))
  \geq \E g(X_{F_2}(z_1),\dots,X_{F_2}(z_n))
\end{displaymath}
for all $z_1,\dots,z_n\in\R^d$, $n\geq1$, and all convex functions
$g:\R^n\to\R$, see \cite{muel:stoy02}. In the case of centred Gaussian
processes, this is equivalent to the fact that the difference of
covariance matrices of finite-dimensional distributions of $X_{F_1}$
and $X_{F_2}$ is positive definite, see
\cite[Sec.~3.13]{muel:stoy02}. By Lemma~\ref{lemma:lem1}, this
holds if and only if $\exp\{-c(\|z\|_{F_1}^{2H}-\|z\|_{F_2}^{2H})\}$
is positive definite for all $c>0$.

\begin{proposition}
  The MfBf $X_{F_1}$ is greater than or equal to the MfBf $X_{F_2}$
  in the convex ordering if and only if $F_1=F_2+_p M$ for an
  $L_p$-ball $M$ and $p=2H$, equivalently, if $\sigma_1=\sigma_2+\nu$
  for a non-negative measure $\nu$, where $\sigma_i$ is the spectral
  measure of $F_i$, $i=1,2$.
\end{proposition}
\begin{proof}
  Note that $\|z\|_{F_1}^{2H}-\|z\|_{F_2}^{2H}$ is a homogenous
  function that can we written as $\|z\|_M^{2H}$ with $F_1=F_2+_pM$ by
  the definition of the $p$-sum, and the $p$-sum of two $L_p$-balls
  corresponds to the arithmetic addition of their spectral measures.
\end{proof}

Sums of independent MfBf's can be interpreted as follows.

\begin{proposition}
  \label{prop:sum}
  Let $X_{F_1}$ and $X_{F_2}$ be two independent MfBf's with
  associated star bodies $F_1$ and $F_2$. Then $X_{F_1}+X_{F_2}$ is
  the MfBf with the associated star body $F=F_1+_p F_2$ being the
  $p$-sum of $F_1$ and $F_2$ for $p=2H$.
\end{proposition}

\begin{corollary}
  Each MfBf with $H\in[\thf,1]$ can be represented as the weak limit
  (on each compact subset of $\R^d$) of the sums of $X_i(A_iz)$,
  $i\geq1$, where $\{X_i,i\geq1\}$ are i.i.d. L\'evy fBf's with
  covariance \eqref{eq:1} and $\{A_i,i\geq1\}$ are positive definite
  matrices.
\end{corollary}
\begin{proof}
  The result follows from Proposition~\ref{prop:sum} and
  \cite[Th.~6.13]{grin:zhan99} saying that each $L_p$-ball $F$ with
  $p\geq1$ can be represented as the limit (in the Hausdorff metric)
  for the $p$-sum of ellipsoids. The MfBf with the associated star
  body being an ellipsoid can be represented as $X(Az)$, $z\in\R^d$,
  where $X$ is the L\'evy fBf. The finite
  dimensional distributions converge, since the convergence of sets in
  the Hausdorff metric yields the convergence of the corresponding
  norms. Finally, the representation as the sum of ellipsoids
  guarantees that at least one summand contains a neighbourhood of the
  origin, so that Proposition~\ref{prop:weak-conv} applies.
\end{proof}

\subsection{Sub-fractional fields}
\label{sec:sub-fract-fields}

Following the definition of the sub-fractional Brownian motion from
\cite{boj:gor:tal04}, it is possible to define its Minkowski analogue
as the centred Gaussian random field with the covariance
\begin{equation}
  \label{eq:sub-frac}
  C_F^{\mathrm{sub}}(z_1,z_2)
  = \|z_1\|_F^{2H} +\|z_2\|_F^{2H} -\thf[\|z_1+z_2\|_F^{2H}+\|z_1-z_2\|_F^{2H}]\,.
\end{equation}
Since
\begin{displaymath}
  C_F^{\mathrm{sub}}(z_1,z_2)
  =C_F(z_1,z_2)+C_F(z_1,-z_2)\,,
\end{displaymath}
\eqref{eq:sub-frac} defines a valid covariance function if $H\in(0,1]$
and $F$ is an $L_p$-ball. The corresponding random field is given by
$\thf(X_F(z)+X_F(-z))$ for the MfBf $X_F$.

The random field $\tilde{X}_F(z)=X_F(z)-X_F(-z)$, $z\in\R^d$, is a
Gaussian random field with the covariance
$\|z_1+z_2\|^{2H}_F-\|z_1-z_2\|^{2H}_F$, whose univariate version was
considered in \cite{boj:gor:tal07}.

\section{Fractional Poisson fields}
\label{sec:fract-poiss-fields}

\subsection{Definition and scaling property}
\label{sec:defin-scal-prop}

For $H\in(0,\thf)$, let $N_H=\{(x_i,r_i),i\geq1\}$ be the
\emph{Poisson point process} (identified with the corresponding
counting measure) on $\R^d\times(0,\infty)$ with the \emph{intensity
  measure}
\begin{equation}
  \label{eq:nuH}
  \nu_H(dx,dr)=dx\, r^{-d-1+2H}dr\,.
\end{equation}
Let $K$ be a convex body in $\R^d$ with non-empty interior.  

\begin{definition}
  \label{def:fbm}
  The \emph{fractional Poisson field} with Hurst index $H$ and the
  shape parameter $K$ is the random field
  \begin{equation}
    \label{eq:pf}
    \xi(z)=\int_{\R^d\times(0,\infty)}(\ind_{z\in x+rK}-\ind_{0\in x+rK})
    \,N_{H}(dx,dr)\,.
  \end{equation}
  Sometimes we write $\xi_K$ or $\xi_{K,H}$ to emphasise the shape
  parameter of the field and its Hurst exponent.
\end{definition}

The random field \eqref{eq:pf} for $K$ being the unit Euclidean ball
was considered in \cite{bier:dem:est13}.  The factor $\lambda$ in
front of the intensity of $\nu_H$ in \cite{bier:dem:est13} can be
incorporated into Definition~\ref{def:fbm} using a rescaled variant of
$K$.

Note that
\begin{align*}
  \int_{\R^d}|\ind_{z\in x+rK}-\ind_{0\in x+rK}|dx
  &=\Vol_d((z+rK)\symdif rK)\\
  &\leq \min(r^d\Vol_d(K),r^{d-1} b_K(z))\,,
\end{align*}
where $\symdif$ denotes the symmetric difference,
$b_K(z)=\Vol_{d-1}(\pr_{z^\perp} K)$ is the $(d-1)$-dimensional volume
of the projection of $K$ onto the hyperplane $z^\perp$ orthogonal to
$z$. Therefore, the integrand in \eqref{eq:pf} belongs to
$L^1(\R^d\times(0,\infty),\nu_{H})$, so that the integral
\eqref{eq:pf} is well defined. Since the absolute difference of two
indicator functions takes values $0$ or $1$, the integrand in
\eqref{eq:pf} also belongs to
$L^2(\R^d\times(0,\infty),\nu_{H})$ and 
\begin{align*}
  \E \xi(z)^2
  &=  \int_{\R^d\times(0,\infty)}(\ind_{z\in rK+x}-\ind_{0\in rK+x})^2\,\nu_{H}(dx,dr)\\
  &= \int_0^\infty\Vol_d((z+rK)\symdif rK)r^{-d-1+2H}\,dr\,.
\end{align*}
Since $\xi(z_1)-\xi(z_2)$ coincides in distribution with
$\xi(z_1-z_2)$, 
\begin{equation}\label{eq:cov-xi}
  \E[\xi(z_1)\xi(z_2)] =\thf
  \bigl[\E\xi(z_1)^2+\E \xi(z_2)^2
  -\E \xi(z_1-z_2)^2\bigr]\,.
\end{equation}

By computing the probability generating functional (see
\cite{dal:ver08}) of the Poisson process $N_H$, it is easy to see that
the finite-dimensional distributions of $\xi$ have the following
characteristic function
\begin{multline}
  \label{eq:char-func}
  \E\exp\left\{\sum_{j=1}^k\imath t_j\xi_K(z_j)\right\}
  =\exp\Bigg\{\int_{\R^d\times(0,\infty)}
    \bigg(\cos\Big(\sum_{j=1}^k t_j\big(\ind_{z_j\in rK+x}
          \\-\ind_{0\in rK+x}\big)\Big)-1\bigg)r^{-d-1+2H}dx\,dr\Bigg\}
\end{multline}

\begin{lemma}
  \label{lemma:lambda}
  For all $a>0$, the random fields $\xi_K(az)$, $z\in\R^d$, and
  $\xi_{bK}(z)$, $z\in\R^d$, with $b=a^{\frac{2H}{d-2H}}$ have
  identical finite-dimensional distributions. The finite-dimensional
  distributions of $\xi_K(az)$ equal the $a^{2H}$-convolution power of
  those of $\xi_K(z)$.
\end{lemma}
\begin{proof}
  The fractional Poisson field with the shape parameter $bK$ equals the
  fractional Poisson field with the shape parameter $K$ generated by the
  point process $\{(x_i,br_i):\; (x_i,r_i)\in N_H\}$. The intensity
  measure of this transformed process is
  \begin{displaymath}
    \E \sum_i \ind_{x_i\in D, br_i\geq t}=\nu_H(D\times
    [b^{-1}t,\infty))
    =a^{2H}\nu_H(D\times [t,\infty))
  \end{displaymath}
  for all Borel $D\subset\R^d$ and $t>0$. Thus, $\xi_{bK}$ equals in
  distribution the fractional Poisson field with the shape parameter
  $K$ and the intensity measure $a^{2H}\nu_H$. This corresponds to a
  superposition of independent Poisson processes and so to the
  convolution power of the distribution.
\end{proof}

\subsection{Relation to the MfBf}
\label{sec:relation-mfbf}

It is shown in \cite{bier:dem:est13} that, if $K$ is the unit
Euclidean ball, the covariance \eqref{eq:cov-xi} of the fractional
Poisson field $\xi$ coincides (up to a multiplicative constant) with
the covariance function \eqref{eq:1} of the L\'evy fBf.  The case of a
general convex body $K$ cannot any longer be handled by the rotational
symmetry argument as in \cite{bier:dem:est13} and, for this, we need
to recall some further concepts from convex geometry.  If $K$ is a
convex body in $\R^d$, then its \emph{radial $p$th mean body} $R_pK$
is defined for $p>-1$ by
\begin{displaymath}
  \|u\|_{R_pK}=\left(\frac{1}{\Vol_d(K)}
    \int_K\rho_K(x,u)^pdx\right)^{-1/p}\,,
\end{displaymath}
where $\rho_K(x,u)=\max\{t:\: x+tu\in K\}$ is the representation of
the boundary of $K$ in the spherical coordinates with the origin
located at $x$, see \cite{gar:zhan98}.

\begin{theorem}
  \label{thr:prm}
  The covariance function of the fractional Poisson field $\xi$ given
  by \eqref{eq:pf} coincides with that of MfBf $X$ with the
  associated star body
  \begin{equation}
    \label{eq:frp}
    F=\left(\frac{H}{\Vol_d(K)}\right)^{1/2H} R_{-2H} K\,.
  \end{equation}
\end{theorem}
\begin{proof}
  Let $u^\perp$ denote the linear space orthogonal to the non-trivial
  vector $u\in\R^d$, and let $\ell_{u,K}(y)$ be the length of the
  segment obtained as the intersection of $K$ with the line
  $\{y+tu:\; t\in\R\}$. Then
  \begin{displaymath}
    \Vol_d((u+K)\symdif K)=
    2\int_{u^\perp} \min(\|u\|,\ell_{u,K}(y))dy\,.
  \end{displaymath}
  If $\|z\|=1$, then
  \begin{equation}\label{eq:var-xi}
  \begin{split}
    \E \xi(z)^2&=\int_0^\infty\Vol_d((z+rK)\symdif rK)r^{-d-1+2H}\,dr\\
    &=\int_0^\infty \Vol_d((zs+K)\symdif K)s^{-1-2H}ds\\
    &=2\int_0^\infty \int_{z^\perp} \min(s,\ell_{z,K}(y))s^{-1-2H}dyds\\
    &=\frac{1}{H(1-2H)} \int_{z^\perp} \ell_{z,K}(y)^{-2H+1}dy\,.
  \end{split}
  \end{equation}
  It is shown in \cite[Lemma~2.1]{gar:zhan98} that, for $p>-1$,
  \begin{displaymath}
    \int_K\rho_K(x,u)^pdx=\frac{1}{p+1}\int_{u^\perp} \ell_{u,K}(y)^{p+1}dy\,.
  \end{displaymath}
  Thus, for $p=-2H$, we have
  \begin{align*}
    \E \xi(z)^2&=
    \frac{1}{H(1-2H)}(-2H+1) \int_K\rho_K(x,z)^{-2H} dx\\
    &=\frac{1}{H} \Vol_d(K) \|z\|_{R_{-2H}K}^{2H}\,.
  \end{align*}
  It remains to note that $\E \xi(tz)^2=t^{2H}\E \xi(z)^2$ for $t>0$.
\end{proof}

\begin{corollary}
  The radial $p$th mean body $R_pK$ of any convex body $K$ is an
  $L_p$-ball for each $p \in(-1,0)$.
\end{corollary}

It is not known if the radial $p$th mean body is convex for
$p\in(-1,0)$ and if two different (up to a translation) convex bodies
share the same radial $p$th mean body for any single $p\in(-1,0)$, see
\cite{gar:zhan98}. 
An inverse to the transform $R_p$ is not yet found.

Below we present a limit theorem that yields the MfBf as a limit
when the intensity of the fractional Poisson field $\xi$ grows and its
argument is rescaled.

\begin{proposition}
  \label{prop:fidi-conv}
  The finite-dimensional distributions of $a^{-H}\xi_K(az)$ converge
  as $a\to\infty$ to those of the MfBf $X_F(z)$, $z\in\R^d$, with the
  associated star body $F$ given by
  \eqref{eq:frp}. Furthermore, $a^{-H}\int_L\xi_K(az)\,dz$
  converges in distribution as $a\to\infty$ to $\int_LX_F(z)\,dz$ for
  each bounded Borel set $L$.
\end{proposition}
\begin{proof}
  By Lemma~\ref{lemma:lambda}, if $a^{2H}=m$ is an integer, then
  $a^{-H}\xi(az)$ is the sum of $m$ i.\,i.\,d.\ copies of $\xi(z)$
  normalised by $\sqrt{m}$. By the central limit theorem, it converges
  to the Gaussian random field that shares the same covariance
  structure with $\xi$, so the MfBf $X$.  A standard argument
  completes the proof of convergence in distribution along an
  arbitrary sequence $a\to\infty$. The weak convergence of integrals
  follows from the central limit theorem and the fact that the
  variances of $\int_L\xi(z)\,dz$ and $\int_L X_F(z)\,dz$ coincide.
\end{proof}

\section{Convergence to MfBf with $H=\thf$}
\label{sec:conv-mfbf-with}

The results from Section~\ref{sec:relation-mfbf} concern the case of
the Hurst parameter $H\in(0,\thf)$. Below we explore the convergence
of the Poisson fractional field to the MfBf with $H=\thf$.

The function $b_K(u)=\Vol_{d-1}(\pr_{u^\perp} K)$, $u\in\Sphere$, is
the support function of the projection body $\Pi K$ to $K$. The
corresponding polar body is denoted by $\Pi^*K$ and called the
\emph{polar projection body}, see \cite[p.~570]{schn2}, so that
$\|u\|_{\Pi^*K}=b_K(u)$ for $u\in\Sphere$.  It is shown in
\cite{gar:zhan98} that the polar projection body $\Pi^*K$ can be
obtained as the limit of $((p+1)\Vol_d(K))^{-1/p} R_pK$ as
$p\downarrow -1$.

\begin{theorem}
  \label{thr:conv1}
  The finite-dimensional distributions of the random field\linebreak
  $\sqrt{1-2H}\,\xi_{K,H}(z)$, $z\in\R^d$, converge as $H\uparrow\thf$
  to the finite-dimensional distributions of the MfBf $X_F(z)$,
  $z\in\R^d$, with the Hurst parameter $H=\thf$ and the associated
  star body $F=\thf\Pi^*K$.
\end{theorem}
\begin{proof}
  Using~\eqref{eq:var-xi}, we have for $\|z\|=1$
  \[
  (1-2H)\E \xi_{K,H}(z)^2
  =\frac{1}{H}\int_{z^\perp}\ell_{z,K}(y)^{-2H+1}dy\,.
  \]
  For $H=\thf$, the integral equals $\|z\|_{\Pi^*K}$, see
  \cite{gar:zhan98}.  Similarly to the proof of Theorem~\ref{thr:prm},
  we obtain for arbitrary $z\in\R^d$ the convergence
  \[
  (1-2H)\E \xi_{K,H}(z)^2\to\|z\|_F \quad\text{as } H\uparrow\thf.
  \]
  By~\eqref{eq:char-func},
  \begin{multline*}
    \E\exp\left\{\sum_{j=1}^k\imath t_j\sqrt{1-2H}\,\xi_{K,H}(z_j)\right\}\\
    =\exp\left\{\int_{\R^d\times(0,\infty)}
      \left(\cos\theta-1\right)r^{-d-1+2H}dx\,dr\right\}
    =\exp\{I_1+I_2\},
  \end{multline*}
  where
  \begin{align*}
    \theta&=\sqrt{1-2H}\sum_{j=1}^k t_j\left(\ind_{z_j\in rK+x}
      -\ind_{0\in rK+x}\right),\\
    I_1&=-\thf\int_{\R^d\times(0,\infty)}\theta^2r^{-d-1+2H}dx\,dr,\\
    I_2&=\int_{\R^d\times(0,\infty)}
    \left(\cos\theta-1+\frac{\theta^2}{2}\right)r^{-d-1+2H}dx\,dr.
  \end{align*}
  Then 
  \begin{align*}
    I_1&=-\frac{1-2H}{2}\int_{\R^d\times(0,\infty)}
    \sum_{j,k}t_jt_k\left(\ind_{z_j\in rK+x}-\ind_{0\in rK+x}\right)\\
    &\quad\times\left(\ind_{z_k\in rK+x}-\ind_{0\in rK+x}\right)r^{-d-1+2H}dx\,dr\\
    &=-\frac{1-2H}{2}\sum_{j,k}t_jt_k\E\left[\xi_{K,H}(z_j)\xi_{K,H}(z_k)\right]\\
    &\to-\frac14\sum_{j,k}t_jt_k\left(\|z_j\|_F+\|z_k\|_F-\|z_j-z_k\|_F\right)
    \qquad \text{as } H\uparrow\thf.
  \end{align*}
  The elementary inequality
  $\left|\cos\theta-1+\frac{\theta^2}{2}\right|\le\frac{|\theta|^3}{6}$
  yields that
  \begin{align*}
    |I_2|&\le\frac16(1-2H)^{3/2}\sum_{i,j,k}|t_it_jt_k|\int_{\R^d\times(0,\infty)}
    \bigl|\left(\ind_{z_i\in rK+x}-\ind_{0\in rK+x}\right)\\
    &\quad\times
    \left(\ind_{z_j\in rK+x}-\ind_{0\in rK+x}\right)
    \left(\ind_{z_k\in rK+x}-\ind_{0\in rK+x}\right)\bigr|
    r^{-d-1+2H}dx\,dr.
  \end{align*}
  The H\"older inequality implies that
  \[
  |I_2|\le\frac{\sqrt{1-2H}}{6}\sum_{i,j,k}|t_it_jt_k|
  \sqrt[3]{\mu_{H}(z_i)\mu_{H}(z_j)\mu_{H}(z_k)},
  \]
  where
  \begin{align*}
    \mu_{H}(z)&=(1-2H)\int_{\R^d\times(0,\infty)}
    \left|\ind_{z\in rK+x}-\ind_{0\in rK+x}\right|^3 r^{-d-1+2H}
    dx\,dr\\
    &=(1-2H)\E\xi_{K,H}(z)^2 \to\|z\|_F \quad \text{as } H\uparrow\thf
  \end{align*}
  for all $z\in\R^d$.
  Thus, 
  $|I_2|\to0$ as $H\uparrow\thf$.
\end{proof}

The integral \eqref{eq:pf} fails to converge if $H\geq \thf$. It is
possible to truncate it to ensure the convergence as follows.  For
$C>0$ and $p>\thf$, define
\[
\eta_{C,p}(z)=\int_{\R^d\times [0,C]}(\ind_{z\in x+rK}-\ind_{0\in x+rK})
\,N_p(dx,dr)\,,
\]
where $N_p$ is the Poisson process with intensity $\nu_H$ from
\eqref{eq:nuH} for $H=p$.  It is easy to see that $\eta_{C,p}$ is well
defined. The following result shows that its normalised version
converges to the MfBf with $H=\thf$ no matter what $p$ is.

\begin{theorem}
  For any $p>\thf$, the finite-dimensional distributions of the random
  field $C^{1/2-p}\eta_{C,p}(z)$, $z\in\R^d$, converge as $C\to\infty$
  to the finite-dimensional distributions of the MfBf with the Hurst
  parameter $H=\thf$ and the associated star body
  $F=\left(p-\thf\right)\Pi^*K$.
\end{theorem}
\begin{proof}
  First let us show that
  \begin{equation}
    \label{eq:var-eta}
    C^{1-2p}\,\E\eta_{C,p}(z)^2\to\|z\|_F\,.
  \end{equation}
  Let $\|z\|=1$.  Similarly to~\eqref{eq:var-xi},
  \[
  C^{1-2p}\,\E\eta_{C,p}(z)^2
  =2C^{1-2p}\int_{1/C}^\infty \int_{z^\perp} \min(s,\ell_{z,K}(y))s^{-1-2p}dyds.
  \]
  Further,
  \begin{align*}
    C^{1-2p}\,&\E\eta_{C,p}(z)^2
    =2C^{1-2p}\int_{z^\perp}\int_{1/C}^\infty
    \ell_{z,K}(y)s^{-1-2p}\ind_{\ell_{z,K}(y)<\frac1C}dsdy\\
    &\qquad+2C^{1-2p}\int_{z^\perp}\int_{1/C}^{\ell_{z,K}(y)}s^{-2p}
    \ind_{\ell_{z,K}(y)\ge\frac1C}dsdy\\
    &\qquad+2C^{1-2p}\int_{z^\perp}\int_{\ell_{z,K}(y)}^\infty
    \ell_{z,K}(y)s^{-1-2p}\ind_{\ell_{z,K}(y)\ge\frac1C}dsdy\\
    &=\frac{C}{p}\int_{z^\perp}\ell_{z,K}(y)\ind_{\ell_{z,K}(y)<\frac1C}dy
    +\frac{2}{2p-1}\int_{z^\perp}\ind_{\ell_{z,K}(y)\ge\frac1C}dy\\
    &\qquad+\frac{C^{1-2p}}{p(2p-1)}\int_{z^\perp}
    \ell_{z,K}(y)^{1-2p}\ind_{\ell_{z,K}(y)\ge\frac1C}dy.
  \end{align*}
  The first term in the right-hand side can be bounded as follows
  \begin{multline*}
    C\int_{z^\perp}\ell_{z,K}(y)\ind_{\ell_{z,K}(y)<\frac1C}dy
    \le\int_{z^\perp}\ind_{\ell_{z,K}(y)<\frac1C}dy\\
    \textstyle
    =V_{d-1}\left(\left\{y\in \pr_{z^\perp}K : \ell_{z,K}(y)<\frac1C\right\}\right)
    \to0 \quad \text{as } C\to\infty.
  \end{multline*}
  The second term converges to $\frac{2}{2p-1}\|z\|_{\Pi^*K}=\|z\|_F$.
  For the third term,
  \begin{align*}
    C^{1-2p}&\int_{z^\perp}\ell_{z,K}(y)^{1-2p}\ind_{\ell_{z,K}(y)\ge\frac1C}dy\\
    &=C^{1-2p}\int_{z^\perp}\ell_{z,K}(y)^{1-2p}
    \ind_{\frac1C\le\ell_{z,K}(y)<\frac{1}{\sqrt C}}dy\\
    &\quad+C^{1-2p}\int_{z^\perp}\ell_{z,K}(y)^{1-2p}
    \ind_{\ell_{z,K}(y)\ge\frac{1}{\sqrt C}}dy\\
    &\le \int_{z^\perp}\ind_{\frac1C\le\ell_{z,K}(y)<\frac{1}{\sqrt C}}dy
    +C^{1/2-p}\int_{z^\perp}\ind_{\ell_{z,K}(y)\ge\frac{1}{\sqrt C}}dy\\
    &\le \int_{z^\perp}\ind_{\ell_{z,K}(y)<\frac{1}{\sqrt C}}dy
    +C^{1/2-p}\|z\|_{\Pi^*K}
    \to0
  \end{align*}
  as $C\to\infty$.  Thus, \eqref{eq:var-eta} holds if $\|z\|=1$.
  Since, for $t>0$,
  \[
  C^{1-2p}\,\E\eta_{C,p}(tz)^2
  =C^{1-2p}t^{2p}\,\E\eta_{C/t,p}(z)^2
  \to t\|z\|_F \quad\text{as }C\to\infty,
  \]
  \eqref{eq:var-eta} holds for arbitrary $z\in\R^d$.

  The convergence of characteristic functions can be verified
  similarly to the proof of Theorem~\ref{thr:conv1}.
\end{proof}

\section{Other constructions of fractional Poisson fields}
\label{sec:other-constr-fract}

The section describes other constructions of Poisson random fields
that share the same covariance structure with the MfBf.

\subsection{Introducing the directional component to the Poisson process}
\label{sec:intr-direct-comp}

Let $F$ be an $L_p$-ball with $p=2H$ for $H\in(0,\thf)$, and let
$\sigma$ be the corresponding spectral measure defined by
\eqref{eq:spectral}.  Consider a Poisson point process $N'_H$ on
$\R\times(0,\infty)\times\Sphere$ with the intensity measure
\begin{displaymath}
  \nu'_H(dx,dr,dv)=dx\, r^{-d-1+2H}dr\,\sigma(dv).
\end{displaymath}
Define the random field
\begin{displaymath}
  \zeta(z)=\int_{\R\times(0,\infty)\times\Sphere}
  \left(\ind_{\langle z,v\rangle\in[x-r,x+r]}-\ind_{0\in[x-r,x+r]}\right)
    \,N'_{H}(dx,dr,dv),
\end{displaymath}
$z\in\R^d$.  Let $\xi(y)$, $y\in\R$, be the univariate fractional
Poisson field with Hurst index $H$ and the shape parameter $K=[-1,1]$.
Then $\E\xi(y)^2=c_H|y|^{2H}$, where $c_H=\frac{2^{1-2H}}{H(1-2H)}$,
see \cite{bier:dem:est13}.  Therefore,
\begin{align*}
  \E\zeta(z)^2
  &=\int_{\R^d\times(0,\infty)\times\Sphere}
  \left(\ind_{\langle z,v\rangle\in[x-r,x+r]}-\ind_{0\in[x-r,x+r]}\right)^2
  \,\nu'_{H}(dx,dr,dv)\\
  &=\int_{\Sphere}\E\xi(\langle z,v\rangle)^2\,\sigma(dv)
  =c_H\int_{\Sphere}|\langle z,v\rangle|^{2H}\,\sigma(dv)
  =c_H\|z\|_F^{2H}.
\end{align*}
Hence,
the covariance function of $\zeta(z)$ is, up to a constant, the
covariance function of MfBf with the associated star body $F$.

\subsection{Poisson processes on Grassmannians}
\label{sec:poiss-proc-grassm}

The affine Grassmannian $A(d,q)$ in $\R^d$ is the family of
$q$-dimensional affine subspaces of $\R^d$, see
\cite[Sec.~13.2]{sch:weil08}. In particular, there is a unique
invariant normalised Haar measure $\mu_q$ on $A(d,q)$, see
\cite[Th.~13.2.12]{sch:weil08}. Let $N_{H,q}=\{(L_i,r_i)\}$ be the
Poisson process on $A(d,q)\times (0,\infty)$ with intensity being the
product measure of $\mu_q$ and the measure with the density
$r^{-d-1+q+2H}$ on $(0,\infty)$. Define
\begin{displaymath}
  \xi(z)=\int_{A(d,q)\times(0,\infty)}
  (\ind_{z\in r_iK+L_i}-\ind_{0\in r_iK+L_i})
  \,N_{H,q}(dL,dr)\,.
\end{displaymath}

By \cite[Eq.~(13.9)]{sch:weil08},
\begin{multline*}
  \int_{A(d,q)}\left|\ind_{z+r\check{K}\cap L\neq\emptyset}
  -\ind_{r\check{K}\cap L\neq\emptyset}\right|\mu_q(dL)\\
  =\int_{G(d,q)}\Vol_{d-q}((\pr_{L^\perp}(rK)+z_{L^\perp})\symdif\pr_{L^\perp}(rK))
  \nu_q(dL)\,,
\end{multline*}
where $\check{K}=\{-x : x\in K\}$, $\nu_q$ is the Haar probability
measure on the Grassmannian $G(d,q)$ (the family of all $q$-dimensional
linear subspaces in $\R^d$), $\pr_{L^\perp}$ denotes the projection on
the subspace $L^\perp$ orthogonal to $L$, and $z_{L^\perp}$ is the
projection of $z$ onto $L^\perp$. The integrand is bounded by a
constant times $\min(r^{d-q},r^{d-1-q})$, so that $\xi(z)$ is well
defined for $H\in(0,\thf)$.

The variance of $\xi(z)$ is given by
\[
  \E\xi(z)^2
  =\int_{G(d,q)} \int_0^\infty
  \Vol_{d-q}((\pr_{L^\perp}(K)+sz_{L^\perp})\symdif\pr_{L^\perp}(K))
  \nu_q(dL) s^{-1-2H} ds\,.
\]
Denote the right-hand side of \eqref{eq:frp} by $F_H(K)$.
Arguing as in the proof of Theorem~\ref{thr:prm}, we obtain that
\begin{align*}
  \E\xi(z)^2 &=\frac{1}{H(1-2H)}\int_{G(d,q)}
  \int_{z^\perp\cap L^\perp} \ell_{z_{L^\perp},\pr_{L^\perp}
    K}(y)^{-2H-1}dy\nu_q(dL)\\
  &=\frac{1}{H}\int_{G(d,q)} \int_{\pr_{L^\perp} K}
  \rho(x,z_{L^\perp})^{-2H} dx\nu_q(dL)\\
  &=\int_{G(d,q)} \|z_{L^\perp}\|^{2H}_{F_H(\pr_{L^\perp}
    K)}\nu_q(dL)\\
  &=\int_{G(d,q)} \|z\|_{F_H(\pr_{L^\perp} K)}\nu_q(dL)=\|z\|_{\tilde F}^{2H}\,.
\end{align*}
The penultimate equation follows from the fact that $F_H(\pr_{L^\perp}
K)$ is given by the sum of $L$ and a subset of $L^\perp$.  The
associated star body $\tilde{F}$ is obtained as the limit of the
$p$-sums (with $p=2H$) of scaled radial $p$th mean bodies of the
projections of $K$. Such MfBf can be viewed as the integral
\begin{displaymath}
  \int_{G(d,q)}\eta_L(z_{L^\perp})\nu_q(dL)
\end{displaymath}
of the independent MfBf's on $L^\perp$ indexed by $L\in G(d,q)$, each
having the associated star body $F_H(\pr_{L^\perp} K)$ being the
scaled radial $p$th mean body of $\pr_{L^\perp} K$.

\section*{Acknowledgements}

IM was supported in part by Swiss National Science Foundation grant
200021-137527.  KR was supported by the Swiss Government Excellence
Scholarship. The authors are grateful to Richard Gardner for advise on
the radial $p$th mean bodies and to Nikolay Leonenko for comments on
the draft version. The comments of the referee have led to numerous
improvements to the presentation. 

\bibliographystyle{model1b-num-names}

\end{document}